\documentclass[smallextended]{svjour3}
\usepackage[utf8]{inputenc}
\usepackage[T1]{fontenc}
\usepackage{enumerate}
\usepackage{stmaryrd}
\usepackage{amsfonts}
\usepackage{ifpdf}
\ifpdf
\usepackage[all,pdf,2cell]{xy}\UseAllTwocells\SilentMatrices
\else
\usepackage[all,xdvi,2cell]{xy}\UseAllTwocells\SilentMatrices
\fi
\usepackage{mathtools}

\newcommand{\EA}{\mathbf{EA}}
\newcommand{\C}{\mathcal{C}}

\newcommand{\Bool}{\mathbf{Bool}}
\newcommand{\Dom}{\mathrm{Dom}}
\newcommand{\id}{\mathrm{id}}

\newcommand{\FinBool}{\mathbf{FinBool}}
\newcommand{\Set}{\mathbf{Set}}

\newcommand{\elems}[1]{\int R(#1)}

\newcommand{\Day}{\otimes_{\mathrm{Day}}}

\newcommand{\obs}[1]{\mathcal O_{#1}}
\begin{document}
\title{Effect algebras as presheaves on finite Boolean algebras}
\author{Gejza Jenča}
\institute{
G. Jenča \at
Department of Mathematics and Descriptive Geometry\\
Faculty of Civil Engineering,
Slovak University of Technology,
	Slovak Republic\\
              \email{gejza.jenca@stuba.sk}
}
\maketitle
\begin{abstract}
For an effect algebra $A$, we examine the category of all morphisms from finite
Boolean algebras into $A$.  This category can be described as a category of
elements of a presheaf $R(A)$ on the category of finite Boolean algebras.  We
prove that some properties (being an orthoalgebra, the Riesz decomposition
property, being a Boolean algebra) of an effect algebra $A$ can be
characterized in terms of some properties of the category of elements of the
presheaf $R(A)$. We prove that the tensor product of effect algebras arises as
a left Kan extension of the free product of finite Boolean algebras along the
inclusion functor. The tensor product of effect algebras can be expressed by
means of the Day convolution of presheaves on finite Boolean algebras.
\end{abstract}
\begin{acknowledgements}
This research is supported by grants VEGA 2/0069/16, 1/0420/15,
Slovakia and by the Slovak Research and Development Agency under the contracts
APVV-14-0013, APVV-16-0073.
\end{acknowledgements}
\keywords{effect algebra, tensor product, presheaf}

\section{Introduction}

In their 1994 paper \cite{FouBen:EAaUQL}, D.J. Foulis and M.K. Bennett defined
effect algebras as (at that point in time) the most general version of
quantum logics. Their motivating example was the set of all Hilbert space
effects, a notion that plays an important role in quantum mechanics
\cite{Lud:FoQM,BusLahMit:TQToM}. An equivalent definition in terms of the difference
operation was independently given by F. Kôpka and F. Chovanec in \cite{KopCho:DP}. Later
it turned out that both groups of authors rediscovered the definition given already in 1989
by R. Giuntini and H. Greuling in \cite{GiuGre:TaFLfUP}.

By the very definition, the class of effect algebras
includes orthoalgebras \cite{FouGreRut:FaSiO}, which include 
orthomodular posets and orthomodular lattices. It soon turned out \cite{ChoKop:BDP}
that there is another interesting subclass of effect algebras, namely MV-algebras
defined by C.C. Chang in 1958 \cite{Cha:AAoMVL} to give the algebraic semantics
of the Łukasiewicz logic. Furthermore, K. Ravindran in his thesis \cite{Rav:OaSToEA} proved that a
certain subclass of effect algebras (effect algebras with the Riesz decomposition property)
is equivalent with the class of partially ordered abelian groups with interpolation \cite{Goo:POAGwI}.
This result generalizes the equivalence of MV-algebras and lattice ordered abelian groups described by
D. Mundici in \cite{Mun:IoAFCSAiLSC}.

In the present paper, we study effect algebras from the viewpoint of category theory.
There are two papers that inspired and motivated the results presented here. 

In their paper \cite{jacobs2012coreflections} Jacobs and Mandemaker utilized
the notion of a coreflective subcategory to prove important results about
effect algebras and their generalized versions.  In particular, they proved
that the category of effect algebras $\EA$ is cocomplete and that 
$\EA$, when equipped with the tensor product
of effect algebras \cite{dvurevcenskij1995tensor}, is a symmetric monoidal category.

In \cite{staton2015effect}, Staton and Uijlen proved that
every effect algebra $A$ can be faithfully represented by a presheaf $R(A)$ on
the category of finite Boolean algebras. This representation is the main tool
we shall use in this paper.

After preliminaries, we prove that several properties (being an orthoalgebra,
having the Riesz decomposition property, being a Boolean algebra) of an effect
algebra $A$ can be characterized by properties of the category of elements of
the representing presheaf $R(A)\colon \FinBool^{op}\to\Set$. 
We use the presheaf representations of effect algebras to prove that the
tensor product of of effect algebras arises as a left Kan extension of the free
product of finite Boolean algebras along the square of the inclusion functor 
of the category of finite Boolean algebras into the category of effect algebras. As a
consequence, the tensor product of effect algebras can be expressed by means of
the Day convolution of presheaves on finite Boolean algebras.

These results mean that the tensor product of effect algebras comes from 
the free product of finite Boolean algebras.  This could be interpreted as an additional
justification of the naturality of the tensor product construction in algebraic
quantum logics.

\section{Preliminaries}

We assume familiarity with basics of category theory, see
\cite{mac1998categories,riehl2016category} for reference.
For effect algebras and related topics, see \cite{DvuPul:NTiQS}. 

\subsection{Effect algebras}

An {\em effect algebra} 
is a partial algebra $(A;+,0,1)$ with a binary 
partial operation $+$ and two nullary operations $0,1$ satisfying
the following conditions.
\begin{enumerate}
\item[(E1)]If $a+b$ is defined, then $b+a$ is defined and
		$a+b=b+a$.
\item[(E2)]If $a+b$ and $(a+b)+c$ are defined, then
		$b+c$ and $a+(b+c)$ are defined and
		$(a+b)+c=a+(b+c)$.
\item[(E3)]For every $a\in A$ there is a unique $a^\perp\in A$ such that
		$a+a^\perp$ exists and $a+a^\perp=1$.
\item[(E4)]If $a+1$ is defined, then $a=0$.
\end{enumerate}

Effect algebras were introduced by Foulis and Bennett in their paper 
\cite{FouBen:EAaUQL}.
In \cite{KopCho:DP}, K\^ opka and Chovanec introduced
an essentially equivalent structure called {\em D-poset}.
Another equivalent structure was introduced by Giuntini and
Greuling in \cite{GiuGre:TaFLfUP}. 

The original definition of an effect algebra
\cite{FouBen:EAaUQL,GiuGre:TaFLfUP} excluded the
case of a one-element effect algebra; it was required that $0\neq 1$.
This has some undesirable consequences: for example a total relation
on an effect algebra is not a congruence in the sense of \cite{GudPul:QoPAM}
and the category of effect algebras lacks the terminal object.
On the other hand, the definition of a D-poset in \cite{KopCho:DP} allows
for one-element D-posets.
In the present paper, we do not assume that $0\neq 1$ in an effect algebra.

In an effect algebra $A$, we write $a\leq b$ if and only if there is
$c\in A$ such that $a+ c=b$.  It is easy to check that for every effect
algebra $A$, $\leq$ is a partial order on $A$.  In this partial order,
$0$ is the smallest and $1$ is the greatest element of the poset $(A,\leq)$,
so every effect algebra has an underlying bounded poset.

The partial operation $+$ is cancellative. Therefore, on every effect algebra it is possible to introduce
a new partial operation $-$; $b-a$ is defined if and only if $a\leq
b$ and then $a+(b-a)=b$.  It can be proved that, in an effect
algebra, $a+ b$ is defined if and only if $a\leq b^\perp$ if and only if $b\leq
a^\perp$. In an effect algebra, we write $a\perp b$ if and only if $a+b$ 
is defined and we say that $a$ and $b$ are {\em orthogonal}.

Let $A_1$, $A_2$ be effect algebras. A map $f\colon A_1\to A_2$ is called a
{\em morphism of effect algebras} if and only if it satisfies the following conditions.
\begin{itemize}
\item $f(1)=1$.
\item If $a\perp b$, then $f(a)\perp f(b)$ and $f(a+ b)=f(a)+f(b)$.
\end{itemize}
Every morphism of effect algebras is an isotone map of the underlying bounded posets. 

A {\em subalgebra} of an effect algebra $A$ is a subset $B\subseteq A$ such that
that $1\in B$ and, for all $x,y\in B$ with $x\geq y$,
$x-y\in B$. Since $x^\perp=1-y$ and $x+y=(x^\perp-y)^\perp$,
every subalgebra is closed with respect to $+$ and $~^\perp$.

The category of effect algebras is denoted by $\EA$. 
The category $\EA$ is complete and cocomplete. The proof of the fact that the category of effect algebras
is cocomplete is nontrivial, see \cite{jacobs2012coreflections} for the proof. Let us point out a
surprising fact the regular epimorphisms in $\EA$ are not necessary surjective,
see \cite[Section 5.2]{jacobs2012coreflections} for an example. This shows that 
the forgetful functor $\EA\to\Set$ that takes the effect algebra to its underlying
set is not monadic, although it is a right adjoint. 
On the other hand, as proved in \cite{jenca2015effect}, the forgetful 
functor that takes an effect algebra to its underlying bounded poset is a 
monadic functor from $\EA$ to the category of bounded posets.

\subsection{Classes of effect algebras, examples}

The class of effect algebras is a common generalization of several types
of algebraic structures.

An effect algebra $A$ is an {\em orthoalgebra} \cite{FouGreRut:FaSiO} if, for
all $a\in A$, $a\perp a$ implies $a=0$. Orthomodular lattices \cite{Kal:OL,Ber:OLaAA} can
be characterized as lattice ordered orthoalgebras.

\begin{example}
Let $\mathcal H$ be a Hilbert space.
The set of all orthogonal projections $P(\mathcal H)$ on $\mathcal H$
is an orthomodular lattice \cite{PtaPul:OSaQL}, hence it is an orthoalgebra.
\end{example}

One can construct examples of effect algebras from an arbitrary partially
ordered abelian group $(G,+,0,\leq)$ in the following way. Choose any positive $u\in
G$; then, for $0\leq a,b\leq u$, define $a\ovee b$ if and only if $a+b\leq u$
and put $a\ovee b=a+b$.  With such partial operation $\ovee$, the interval
$$
[0,u]_G=\{x\in G\colon 0\leq x\leq u\}
$$ becomes an effect algebra $([0,u]_G,\ovee,0,u)$.  Effect algebras that
arise from partially ordered abelian groups in this way are called {\em interval
effect algebras}, see \cite{BenFou:IaSEA}.

\begin{example}
The closed real interval $[0,1]_{\mathbb R}$ is an interval effect algebra. 
\end{example}
\begin{example}
Let $\mathcal H$ be a Hilbert space. Let $S(\mathcal H)$ be the set of
all bounded self-adjoint operators on $\mathcal H$. 
For $A,B\in S(\mathcal H)$, write $A\leq B$ if and only if $B-A$ has a nonnegative spectrum.
Then $S(\mathcal H)$ is a partially ordered abelian group.
The interval $E(\mathcal H)=[0,I]_{S(\mathcal H)}$, where $I$ is the identity operator,
is an interval effect algebra,
called the {\em standard effect algebra}.
\end{example}

\subsection{Finite summable families}

Let $A$ be an effect algebra. For a finite set $I$, an ($I$-indexed)
{\em summable family} of elements of $A$ is a family $(a_i)_{i\in I}$ such that the
sum
$\sum_{i\in I}a_i$
exists in $A$. A finite summable family $(a_i)_{i\in I}$ with $\sum_{i\in I}a_i=1$ is called a
{\em finite decomposition of unit}.

We say that a summable family $(b_j)_{j\in J}$ is a {\em refinement}
of a summable family $(a_i)_{i\in I}$ if there is a surjective mapping
$\rho\colon J\to I$ such that, for all $i\in I$,
$$
a_i=\sum_{\rho(j)=i} b_j.
$$
It is easy to see that if $(b_j)_{j\in J}$ is a refinement of $(a_i)_{i\in I}$, then
$\sum_{i\in I}a_i=\sum_{j\in J}b_j$.

\subsection{Boolean algebras, observables}
\label{subs:boolean}
Every Boolean algebra $(A,\vee,\wedge,~^\perp,0,1)$ is an effect algebra.
The partial operation $+$ on $A$ is given by the rule $x\perp y$ if and only if
$x\wedge y=0$ and then $x+y=x\vee y$. Clearly, this is an effect algebra with
the same partial order as the original Boolean algebra. This shows that the
category of Boolean algebras $\Bool$ is a subcategory of the category of effect
algebras $\EA$. Moreover, $\Bool$ is a full subcategory of $\EA$.  Therefore, we can (and
we will) identify Boolean algebras with their respective effect-algebraic
versions.

If $X$ is a Boolean algebra and $A$ is an effect algebra, then a morphism
$g\colon X\to A$ is called an {\em ($A$-valued) observable}. 
In general, the range of an observable $g\colon X\to A$ is not 
a sub-effect algebra of $A$. However, if $A$ is an orthoalgebra
then the range of $g$ is a sub-effect algebra of $A$.
Moreover, $g(X)$ is then a Boolean algebra.

\subsection{A notation for finite observables}

In what follows, we abbreviate the initial segment of natural numbers
$\{1,\dots,n\}$ by $[n]$. Note that $[0]=\emptyset$.

An observable from a finite Boolean algebra to an effect algebra is called
a {\em finite observable}.  If $A$ is an effect algebra and
$g\colon 2^{[n]}\to A$ is a finite $A$-valued observable, then it is obvious that $g$ is
determined by its values on singleton subsets of $[n]$. Indeed, 
every $X=\{x_1,\dots,x_k\}\in 2^{[n]}$ can be expressed as 
a sum of singletons $X=\{x_1\}+\dots+\{x_k\}$ and hence 
$$
g(X)=g(\{x_1\}+\dots+\{x_k\})=g(\{x_1\})+\dots+g(\{x_k\}).
$$
Thus, $g$ can be expressed by a finite $[n]$-indexed decomposition of unit 
$$
(g(\{x_1\}),\dots,g(\{x_n\})).
$$
Note that we can safely omit the last element
of this sequence without losing any information, because
$$
g(\{x_n\})=\bigl(\,g(\{x_1\})+\dots+g(\{x_{n-1}\})\,\bigr)^\perp.
$$
In this way, every summable sequence $(a_1,\dots,a_{n-1})$ of elements of an effect algebra $A$ 
determines a finite observable
$\obs{a_1,\dots,a_{n-1}}2^{[n]}\colon ^{[n]}\to A$ and every finite $A$-valued observable on $2^{[n]}$
is determined by a summable sequence of $n-1$ elements of $A$.
We will use the notation $\obs{a_1,\dots,a_{n-1}}$ throughout this paper.

For example, for every element $a\in A$, $\obs{a}$ denotes the observable $2^{[2]}\to A$ given by 
the table
$$
\begin{array}{|c||c|c|c|c|}
\hline
X& \emptyset & \{1\} & \{2\} & \{1,2\} \\
\hline
\obs{a}(X)&0 & a & a^\perp & 1\\
\hline
\end{array}
$$
The symbol $\obs{\emptyset}$ denotes the (unique) observable $\obs{\emptyset}\colon 2^{[1]}\to A$.
Note that $\obs{0}$ and $\obs{\emptyset}$ are not the same thing.

\subsection{Riesz decomposition property}

An effect algebra $A$ satisfies the {\em Riesz decomposition property} if and
only if, for all $u,v_1,v_2\in A$, $u\leq v_1+v_2$ implies that there exist
$u_1,u_2\in A$ such that $u_1\leq v_1$, $u_2\leq v_2$ and $u=u_1+u_2$. It was
proved in \cite{Rav:OaSToEA} that every effect algebra satisfying the Riesz
decomposition property is an interval in a partially ordered abelian group
satisfying the Riesz decomposition property. Such groups are
sometimes called {\em interpolation groups}, see \cite{Goo:POAGwI}. Every
lattice-ordered abelian group is an interpolation group.

\begin{example}
The set of all differentiable functions $\mathbb R\to [0,1]_{\mathbb R}$
forms an effect algebra satisfying the Riesz decomposition property. We note
that this effect algebra is not lattice ordered.
\end{example}

\begin{proposition}
For an effect algebra $A$, the following are equivalent
\begin{enumerate}[(a)]
\item $A$ satisfies the Riesz decomposition property.
\item For all $x_1,\dots,x_n,y_1,\dots,y_m\in A$ such that
$x_1+\dots+x_n=y_1+\dots+y_m$ there exists an $n\times m$ matrix
$Z=(z_{ij})$ of elements of $A$ such that, for all
$i=1,\dots,n$, $x_i$ is the sum of $i$-th row and, for
all $j=1,\dots,m$, $y_j$ is the sum of $j$-th column of $Z$.
\item $A$ satisfies (b) for $n=m=2$.
\end{enumerate}
\end{proposition}
\begin{proof}
See \cite[Section 1.7]{DvuPul:NTiQS}.
\end{proof}

In our terminology, we may express this as follows.
\begin{proposition}
An effect algebra $A$ satisfies the Riesz decomposition property if and only if
any two summable families with the same sum admit a common refinement.
\end{proposition}

It follows from the main result of \cite{ChoKop:BDP} that there
is a one-to-one correspondence between lattice ordered effect algebras satisfying
the Riesz decomposition property and MV-algebras, introduced by Chang 
\cite{Cha:AAoMVL} in the 1950s to give an algebraic counterpart of the many-valued Łukasiewicz
logic. It was proved by Mundici in \cite{Mun:IoAFCSAiLSC} that 
every MV-algebra is an interval in a lattice-ordered
abelian group and vice versa.

\subsection{Stone duality for finite Boolean algebras}

Recall, that the category of finite Boolean algebras is dually equivalent to the category of finite sets.
Explicitly, if $t\colon [m]\to [n]$ is a mapping of finite sets, then a dual morphism of Boolean
algebras
$\widehat t\colon 2^{[n]}\to 2^{[m]}$ is given by the rule $\widehat t(X)=t^{-1}(X)=\{j\in[m]:t(j)\in X\}$.
If $f\colon 2^{[n]}\to 2^{[m]}$ is a morphism of Boolean algebras, then for
every $j\in [m]$ there is exactly one $i\in[n]$ such that $j\in f(\{i\})$; this
$i$ is then the value of the dual map $\widehat f(j)$.

Via this duality, 
the coproduct in the category of finite Boolean algebras (denoted by $*$) is
dual to the product of finite sets.
Thus, we may exhibit $2^{[n]}*2^{[m]}$ as $2^{[n]\times[m]}$. If 
$s\colon 2^{[n]}\to 2^{[n']}$ and $t\colon 2^{[m]}\to 2^{[m']}$
are morphisms of Boolean algebras, the mapping
$s*t\colon 2^{[n]\times[m]}\to 2^{[n']\times[m']}$ is then given by the
rule
\begin{equation}
\label{eq:coproduct}
(s*t)(X)=\bigcup_{(i,j)\in X}s(\{i\})\times t(\{j\}).
\end{equation}
For our purposes, it is important to note that the sets occurring in the union in (\ref{eq:coproduct}) are pairwise disjoint. Indeed,
if $(k,l)\in (s(\{i_1\})\times t(\{j_1\}))\cap(s(\{i_2\})\times t(\{j_2\}))$
then $k\in s(\{i_1\})\cap s(\{i_2\})$ and $l\in t(\{j_1\})\cap t(\{j_2\})$
and this already implies that $i_1=i_2$ and $j_1=j_2$. Therefore, we may write
the union in (\ref{eq:coproduct}) as an effect-algebraic sum:
\begin{equation}
\label{eq:coproduct2}
(s*t)(X)=\sum_{(i,j)\in X}s(\{i\})\times t(\{j\}).
\end{equation}

\subsection{Bimorphisms, tensor products}

For effect algebras $A,B$ and $C$ a mapping $h\colon A\times B\to C$ is a $C$-valued {\em bimorphism}
\cite{dvurevcenskij1995tensor}
from $A,B$ to $C$ if and only if the following conditions are satisfied.
\begin{description}
\item[\bf Unitality:] $h(1,1)=1$.
\item[\bf Left additivity:] For all $b\in B$ and $a_1,a_2\in A$ such that $a_1\perp a_2$, $h(a_1,b)\perp h(a_2,b)$ and
$h(a_1,b)+h(a_2,b)=h(a_1+a_2,b)$.
\item[\bf Right additivity:] For all $a\in A$ and $b_1,b_2\in B$ such that $b_1\perp b_2$, $h(a,b_1)\perp h(a,b_2)$ and
$h(a,b_1)+h(a,b_2)=h(a,b_1+b_2)$.
\end{description}

It is easy to check that for every morphism of effect algebras $f\colon C\to C'$
and a bimorphism $h\colon A\times B\to C$, $f\circ h$ is a bimorphism. This fact shows that there
is a category $\beta_{A,B}$
where the objects are all bimorphisms from $A,B$ and the morphisms are $\EA$-morphisms acting on bimorphisms
from left by composition.

\begin{definition}
\cite{dvurevcenskij1995tensor}
Let $A,B$ be effect algebras. A tensor product of $A$ and $B$ 
(denoted by $A\otimes B$) is the initial object in the category $\beta_{A,B}$.
\end{definition}

The notions of bimorphism and of the tensor product of orthoalgebras were given by
Foulis and Bennett in \cite{foulis1993tensor}. It was proved by Jacobs and Mandemaker
in \cite{jacobs2012coreflections} that the category of effect algebras equipped with
the tensor products forms a symmetric monoidal category. There is another important result concerning tensor
products:
in \cite{borger2004tensor} B\"orger proved that orthomodular posets equipped
with tensor product form a symmetric monoidal category. Let us remark that B\"orger's proof
applies, almost without changes, in the more general case of effect algebras.

In the paper \cite{dvurevcenskij1995tensor}, it was assumed that $0\neq 1$ in every effect algebra.
Consequently, it might happen that there are pairs $A,B$ of effect algebras such that there is
no bimorphism $h\colon A\times B\to C$, so $A\otimes B$ does not exist. However, if we allow
for one-element effect algebras, then tensor product of effect algebras always exists and if
$A\otimes B$ has more than one element, it coincides with the tensor product as defined
in \cite{dvurevcenskij1995tensor}.

Thus, ``our'' tensor products are the same as the tensor products in the sense of \cite{dvurevcenskij1995tensor},
whenever the tensor product exists in the sense of \cite{dvurevcenskij1995tensor},
and we obtain $A\otimes B=\{0\}$ whenever tensor product does not exist in the sense of \cite{dvurevcenskij1995tensor}.

\section{Presheaves on finite Boolean algebras}

For a general background for this section, see \cite[Section I.5]{maclane2012sheaves}.

Let $\FinBool$ be the full subcategory of the category of Boolean algebras
$\Bool$ spanned by the set of objects $\{2^{[n]}\colon n\in\mathbb N\}$. The
restriction of the fully faithful functor $\Bool\to\EA$ described in the
subsection \ref{subs:boolean} to the subcategory $\FinBool$ gives us a fully faithful
functor $E\colon \FinBool\to\EA$.

The functor $R\colon \EA\to[\FinBool^{op},\Set]$ maps every effect algebra $A$ to a presheaf
$R(A)\colon \FinBool^{op}\to\Set$. The presheaf $R(A)$ maps a Boolean algebra $2^{[n]}$ to the set of
all $A$-valued observables on $2^{[n]}$:
$$
R(A)(2^{[n]})=\EA(E(2^{[n]}),A).
$$
For every morphism of Boolean algebras $f\colon 2^{[n]}\to 2^{[m]}$,
$$
R(A)(f)\colon \EA(E(2^{[m]}),A)\to\EA(E(2^{[n]}),A)
$$ is given by the rule
$(R(A)(f))(g)=g\circ E(f)$.

Recall, that for a category $\C$ and a 
presheaf $P\colon \C^{op}\to\Set$, the category of elements $\int P$ of $P$ is 
a category defined as follows:
\begin{itemize}
\item Objects are all pairs $(C,g)$, where $C$ is an object of $\C$ and $g\in P(C)$.
\item An arrow $(C,g)\to (C',g')$ is an arrow $f\colon C\to C'$ in $\C$ such that $P(f)(g')=g$.
\end{itemize}

For an effect algebra $A$, the category $\elems{A}$ is the {\em category of finite observables},
which can be explicitly described as follows:
\begin{itemize}
\item Objects are all pairs $(2^{[n]},g)$, where $g\colon E(2^{[n]})\to A$ is an observable.
\item An arrow $(2^{[n]},g)\to (2^{[n']},g')$ is a morphism of Boolean algebras $f\colon 2^{[n]}\to 2^{[n']}$
such that $g'\circ E(f)=g$.
\end{itemize}
Since $\FinBool$ is small and $\EA$ is locally small, $\elems{A}$ is small.

Note that the first component of every pair $(2^{[n]},g)\in\elems{A}$ contains 
redundant information, because $2^{[n]}$ is the domain of $g$.
Therefore, we shall mostly write simply $g$ instead of $(2^{[n]},g)$ whenever
there is no danger of confusion.
Furthermore, since $\FinBool$ is a full subcategory of $\EA$, we shall mostly 
suppress the functor $E$ from our notations. We shall write, for example, 
$g\colon 2^{[n]}\to A$ instead of $g\colon E(2^{[n]})\to A$ and $g'\circ f=g$ instead of
$g'\circ E(f)=g$.

For every presheaf $P\colon\FinBool^{op}\to\Set$, there is a projection functor
$\pi_P\colon\int P\to\FinBool$ given by $\pi_P(2^{[n]},g)=2^{[n]}$. By a general argument
\cite[Theorem I.5.2]{maclane2012sheaves} the functor 
$L\colon [\FinBool^{op},\Set]\to\EA$ given by the colimit
$$
L(P)=\varinjlim\left(\int P\xrightarrow{\pi_P}\FinBool\xrightarrow{E}\EA\right)
$$
is left adjoint to $R$.

For an effect algebra $A$, $D_A$ denotes the functor
$$
\int R(A)\xrightarrow{\pi_{R(A)}}\FinBool\xrightarrow{E}\EA.
$$
Note that $L(R(A))\simeq\varinjlim D_A$.

The following theorem was stated by Staton and Uijlen in \cite{staton2015effect}.
To keep out presentation self-contained, we give a complete proof.

\begin{theorem}
\label{thm:density}

The adjunction 
$$
\xymatrix{
[\FinBool^{op},\Set]
	\ar@/^1.2pc/[rr]^-L
&
\hspace*{-2em}\bot
&
\EA
	\ar@/^1.2pc/[ll]^-R
}
$$
 is a reflection.
 \end{theorem}
\begin{proof}
We need to prove that, for every effect algebra $A$ $L(R(A))\simeq A$,
that means, that $A$ is a colimit of the functor $D_A$.

It is clear that the objects of 
$\elems{A}$ indexed by themselves form a cocone with apex $A$ under $D_A$. We claim
that this cocone is initial in the category of cocones under $D_A$. We need to prove that for every other
cocone $(r_g,V)$ under $D_A$ with apex $V$ consisting of a family of 
$r_{g}$, where $(2^{[n]},g)$ runs through all objects of $\elems{A}$,
there is a unique morphism of effect algebras $u\colon A\to V$ such that $r_{g}=u\circ g$
for every object of the category $\elems{A}$.

This property already determines the only possible candidate mapping for the morphism $u\colon A\to V$. Indeed,
for $n=2$ and $g=\obs{a}$ we must have
$r_{\obs{a}}=u\circ \obs{a}$, in particular,
$$
r_{\obs{a}}(\{1\})=u(\obs{a}(\{1\}))=u(a),
$$
and we see that $u(a)=r_{\obs{a}}(\{1\})$. 
We claim that this $u\colon A\to V$ is a morphism of effect algebras and that for
every observable $g\colon 2^{[n]}\to A$ we have $r_{g}=u\circ g$.

Let $a_1,a_2\in A$ be such that $a_1+a_2$ exists in $A$. Consider the observable 
$\obs{a_1,a_2}\colon 2^{[3]}\to A$.
There are three unique morphisms $f_1,f_2,f_{1,2}\colon 2^{[2]}\to2^{[3]}$ of Boolean algebras
that make the three triangles in the diagram
\begin{equation}
\label{diag:summation}
\xymatrix{
~
&
2^{[2]}
	\ar[ld]_{f_1}
	\ar[rd]^{\obs{a_1}}
\\
2^{[3]}
	\ar[rr]^{\obs{a_1,a_2}}
&
~
&
A
\\
&
2^{[2]}
	\ar[lu]^{f_2}
	\ar[ru]_{\obs{a_2}}
\\
~
&
2^{[2]}
	\ar@/^1.2pc/[luu]^{f_{1,2}}
	\ar@/_1.2pc/[ruu]_{\obs{a_1+a_2}}
}
\end{equation}
commute. Explicitly, 
$f_1(\{1\})=\{1\}$, $f_2(\{1\})=\{2\}$ and $f_{1,2}(\{1\})=\{1,2\}$.

The commutativity of (\ref{diag:summation}) means that $f_1,f_2,f_{1,2}$
can be considered as arrows in the category $\elems{A}$:
\begin{align*}
f_1\colon\obs{a_1}&\to\obs{a_1,a_2}\\
f_2\colon\obs{a_2}&\to\obs{a_1,a_2}\\
f_{1,2}\colon\obs{a_1+a_2}&\to\obs{a_1,a_2}.
\end{align*}

Therefore, since $(r_g,V)$ is a cocone under $D_A$, we may compute
\begin{align*}
u(a_1+a_2)=
r_{\obs{a_1+a_2}}(\{1\})=
(r_{\obs{a_1,a_2}}\circ f_{1,2})(\{1\})&=\\
r_{\obs{a_1,a_2}}(\{1,2\})=r_{\obs{a_1,a_2}}(\{1\}+\{2\})&=\\
r_{\obs{a_1,a_2}}(\{1\})+r_{\obs{a_1,a_2}}(\{2\})&=\\
(r_{\obs{a_1,a_2}}\circ f_1)(\{1\})+(r_{\obs{a_1,a_2}}\circ f_2)(\{1\})&=\\
r_{\obs{a_1}}(\{1\})+r_{\obs{a_2}}(\{1\})&=u(a_1)+u(a_2)
\end{align*}
and we see that $u$ preserves $+$.

To prove that $u(1)=1$, consider the unique observable $\obs{\emptyset}\colon 2^{[1]}\to A$ and
an arrow $z\colon 2^{[2]}\to 2^{[1]}$ 
given by $z(\{1\})=\{1\}$, $z(\{2\})=\emptyset$. From the commutativity of
$$
\xymatrix{
2^{[1]}
	\ar[r]^{\obs{\emptyset}}
&
A
\\
2^{[2]}
	\ar[u]^{z}
	\ar[ru]_{\obs{1}}
}
$$
in $\EA$ it follows that
$z$ is a morphism
$z\colon \obs{1}\to \obs{\emptyset}$ in $\elems{A}$, hence 
\begin{align*}
u(1)=r_{\obs{1}}(\{1\})=(r_{\obs{\emptyset}}\circ z)(\{1\})=\\
r_{\obs{\emptyset}}(\{1\})=1.
\end{align*}

It remains to prove that  
$u$ 
is a morphism of cocones under $D_A$, that means,
$r_{g}=u\circ g$ 
for every object $g$ of the category $\elems{A}$.
Let $g\colon 2^{[n]}\to A$ and let $X=\{x_1,\dots,x_k\}\in 2^{[n]}$.
For every $i\in[k]$, let $f_i\colon (2^{[2]},\obs{g(\{x_i\})})\to (2^{[n]},g)$ be
a morphism in $\elems{A}$ given by the rules $f_i(\{1\})=\{x_i\}$,
$f_i(\{2\})=[n]\setminus\{x_i\}$. Then,
\begin{align*}
r_{g}(X)=r_{g}(\sum_{i=1}^k\{x_i\})=\\
\sum_{i=1}^k r_{g}(\{x_i\})=
\sum_{i=1}^k(r_{g\circ f_i}(\{1\}))=\\
\sum_{i=1}^k u(g(\{x_i\}))=
\sum_{i=1}^k u(g(\{x_i\}))=\\
u(g(\sum_{i=1}^k\{x_i\}))=u(g(X)).
\end{align*}
\end{proof}
We say that a category $\C$ is {\em amalgamated} if and only if 
every span in $\C$ can be extended to a commutative square.
\begin{theorem}\label{thm:rdpisamalgamated}
An effect algebra $A$ satisfies the Riesz decomposition property if and only if 
$\elems{A}$ is amalgamated.
\end{theorem}
\begin{proof}
Suppose that $A$ satisfies the Riesz decomposition property. Let $g,g_1,g_2$ be
$A$-valued finite observables, let $f_1\colon g\to g_1$ and $f_2\colon g\to g_2$. Write
$g\colon 2^{[n]}\to A$; for every $i\in [n]$ consider the sets $f_1(\{i\}),f_2(\{i\})$. It is easy
to see that 
$$
(g_1(\{j\}))_{j\in f_1(\{i\})}
\qquad
(g_2(\{k\}))_{k\in f_2(\{i\})}
$$
are both finite summable families with sum equal to $g(i)$. By the Riesz
decomposition property, these families have a common refinement, let
us call it 
$$
(w^i_{(j,k)})_{(j,k)\in f_1(\{i\})\times f_2(\{j\})}.
$$ 
Concatenating all these families $(w^i_{(.,.)})$ gives us a decomposition
of unit that is easily seen to be a common refinement of the decompositions
of unit associated with the observables $g_1$ and $g_2$; let us denote
the observable associated with the common refinement by $z$.
There are morphisms $h_1\colon g_1\to z$, $h_2\colon g_2\to z$ in $\elems{A}$ associated
with the refinements of $g_1,g_2$
and, obviously, $h_1\circ f_1=h_2\circ f_2$.

Suppose that $\elems{A}$ is amalgamated. Let $u,x_1,x_2,y_2,y_2\in A$ be such
that $x_1+x_2=y_1+y_2=u$. Consider the $A$-valued finite observables
$\obs{u},\obs{x_1,x_2},\obs{y_1,y_2}$ associated with the decompositions of unit
$(u,u')$, $(x_1,x_2,u')$ and $(y_1,y_2,u')$ and equip them with the natural
arrows $f_x\colon \obs{u}\to \obs{x_1,x_2}$ and $f_y\colon \obs{u}\to \obs{y_1,y_2}$.  Since
$\elems{A}$ is amalgamated, the span $\obs{x_1,x_2},\obs{y_1,y_2}$ extends to a
commutative square, so there is an $A$-valued finite observable $g$ and
morphisms of observables $h_x\colon \obs{x_1,x_2}\to g$ and $h_y\colon \obs{y_1,y_2}\to g$ such
that $h_x\circ f_x=h_y\circ f_y$. It is easy to check that $h_x,h_y$ give us the
desired common refinement of the summable sequences $(x_1,x_2)$ and $(y_1,y_2)$.
\end{proof}
\begin{theorem}
\label{thm:oacoequalizes}
An effect algebra $A$ is an orthoalgebra if and only if
for every pair of morphisms $f_1,f_2\colon g\to g'$ in $\elems{A}$
there is a coequalizing morphism $q\colon g'\to u$ such that $q\circ f_1=q\circ f_2$. 
\end{theorem}
\begin{proof}
Suppose that $A$ is an orthoalgebra. 
Let $g\colon 2^{[n]}\to A$ and
$g'\colon 2^{[m]}\to A$ be finite observables, let $f_1,f_2\colon g\to g'$ in $\obs{A}$.
Since $A$ is an orthoalgebra, the range of every $A$-valued
observable is a Boolean subalgebra of $A$. Therefore, there exists
a Boolean algebra $B$ (for example the range of $g'$) and an embedding $j\colon B\to A$ such that $g,g'$
factor through $j$. That means, there are morphisms of effect algebras $h,h'$ such that the diagram
\begin{equation}
\xymatrix@C=2em@R=3em{
2^{[n]}
	\ar@<-.5ex>[dd]_{f_1}
	\ar@<.5ex>[dd]^{f_2}
	\ar[rd]^h
	\ar[rrrd]^g
&
~
&
~
\\
~
&
B
	\ar@{^{(}->}[rr]^j
&
~
&
A
\\
2^{[m]}
	\ar[ru]_{h'}
	\ar[rrru]_{g'}
}
\end{equation}
commutes in $\EA$.

Let $q\colon 2^{[m]}\to 2^{[k]}$ be a coequalizer of 
the pair $f_1,f_2$ in the category $\Bool$. As $q$ is a coequalizer and 
$h=h'\circ f_1=h'\circ f_2$ in $\Bool$, there is a (unique) morphism of Boolean algebras
$u\colon 2^{[k]}\to B$ such that $h'=u\circ q$, hence the diagram 
\begin{equation}
\xymatrix@C=4em{
2^{[n]}
	\ar@<-.5ex>[d]_{f_1}
	\ar@<.5ex>[d]^{f_2}
	\ar[rd]^h
&
~
\\
2^{[m]}
	\ar[r]_{h'}
	\ar[d]_q
&
B
\\
2^{[k]}
	\ar[ru]_u
}
\end{equation}
commutes in $\EA$.
Therefore, $g'=j\circ h'=j\circ u\circ q$ or,
in other words, $q\colon g'\to j\circ u$ is the arrow in $\elems{A}$ with the property $q\circ f_1=q\circ f_2$.

Suppose that $A$ is an effect algebra such that 
for every pair of morphisms $f_1,f_2\colon g\to g'$ in $\elems{A}$
there is a morphism $q\colon g'\to h$ such that $q\circ f_1=q\circ f_2$. 
Let $a\in A$ be such that $a\perp a$. We need to prove that $a=0$.
Let $f_1,f_2\colon 2^{[2]}\to 2^{[3]}$ be such that $f_1(\{1\})=\{1\}$ and
$f_2(\{1\})=\{2\}$. Then $\obs{a}=\obs{a,a}\circ f_1=\obs{a,a}\circ f_2$ in $\EA$, that means,
$f_1,f_2\colon \obs{a}\to \obs{a,a}$ in $\elems{A}$. By assumption, there is an arrow
$q\colon 2^{[3]}\to 2^{[n]}$ such that $q\circ f_1=q\circ f_2$ 
and an observable $u\colon 2^{[n]}\to A$ such that $u\circ q=\obs{a,a}$.
Thus, the following diagram commutes in $\EA$:
\begin{equation}
\xymatrix@C=4em{
2^{[2]}
	\ar@<-.5ex>[d]_{f_1}
	\ar@<.5ex>[d]^{f_2}
	\ar[rd]^{\obs{a}}
&
~
\\
2^{[3]}
	\ar[r]_{\obs{a,a}}
	\ar[d]_q
&
A
\\
2^{[n]}
	\ar[ru]_u
}
\end{equation}
This implies that $q(\{1\})=q(f_1(\{1\}))=q(f_2(\{1\}))=q(\{2\})$. On the other hand,
since $\{1\}\perp\{2\}$ in $2^{[3]}$, $q(\{1\})\perp q(\{2\})$ in $2^{[n]}$. Since
$2^{[n]}$ is a Boolean algebra, it is an orthoalgebra,
hence $q(\{1\})\perp q(\{2\})$ and $q(\{1\})=q(\{2\})$ imply 
$q(\{1\})=q(\{2\})=0$. Finally,
$$
a=\obs{a}(\{1\})=\obs{a,a}(f_1(\{1\}))=\obs{a,a}(\{1\})=u(q(\{1\}))=u(0)=0.
$$
\end{proof}
Recall, that a category $\C$ is called {\em filtered} if and only if the
following conditions are satisfied.
\begin{itemize}
\item $\C$ is nonempty.
\item For every pair of objects $X_1,X_2$ there is a cospan 
$$
\xymatrix{
~
&
X
&
~\\
X_1 
	\ar[ru]
& 
~ 
& 
X_2
	\ar[lu]
}
$$
over them.
\item 
For every parallel pair of
morphisms $f_1,f_2:X\to Y$ in there exists a morphism $q:Y\to Z$ such that $q\circ f_1=q\circ f_2$.
\end{itemize}

\begin{corollary}
An effect algebra $A$ is a Boolean algebra if and only if $\elems{A}$ is filtered.
\end{corollary}
\begin{proof}
An effect algebra $A$ is a Boolean algebra if and only if $A$ satisfies the
Riesz decomposition property and $A$ is an orthoalgebra.  The rest of the proof
follows easily by Theorem \ref{thm:rdpisamalgamated} and Theorem
\ref{thm:oacoequalizes} using the fact that $\elems{A}$ has an initial object
$\obs{\emptyset}\colon 2^{[1]}\to A$.
\end{proof}

\section{Tensor products}

Let $A$, $B$ be effect algebras. The category $\elems{A}\times\elems{B}$ has pairs of finite
observables as objects and pairs of morphisms of observables as arrows. Consider the functor
$D_{A,B}\colon \elems{A}\times\elems{B}\to\EA$ given by the rule
$$
D_{A,B}(g_A,g_B)=\Dom(g_A)*\Dom(g_B),
$$
where $*$ denotes free product (that means, coproduct in $\Bool$) of Boolean algebras.

\begin{lemma}
Let $A,B$ be effect algebras.
\label{lemma:tensors:colimits}
The category of bimorphisms $\beta_{A,B}$ is isomorphic to the category of cocones under the diagram $D_{A,B}$.
Under this isomorphism, $C$-valued bimorphisms correspond to cocones with apex $C$ and vice versa. 
\end{lemma}
\begin{proof}
We shall describe how to construct a cocone under $D_{A,B}$ with apex $C$ from a $C$-valued bimorphism
and vice versa so that the constructions are mutually inverse.

Let $h\colon A\times B\to C$ be a bimorphism. 
We need to construct a cocone under $D_{A,B}$ associated to $h$.
So for every pair of finite observables $g_A\colon 2^{[n]}\to A$ and
$g_B\colon 2^{[m]}\to B$, we need to define a morphism $v_{g_A,g_B}\colon 2^{[n]}*2^{[m]}\to C$, that will be
the component of our cocone at $(g_A,g_B)$. Note that $2^{[n]}*2^{[m]}\simeq 2^{[n]\times[m]}$ and
$$
1=h(1,1)=h(\sum_{i\in[n]}g_A(\{i\}),\sum_{j\in[m]}g_B(\{j\}))=
\sum_{\substack{i\in[n]\\j\in[m]}}h(g_A(\{i\}),g_B(\{j\})),
$$
hence $(h(g_A(\{i\}),g_B(\{j\})))_{(i,j)\in[n]\times[m]}$ is a
$[n]\times[m]$-indexed decomposition of unit in $C$.
Therefore, $v_{g_A,g_B}\colon 2^{[n]\times[m]}\to C$ given by
$$
v_{g_A,g_B}(X)=\sum_{(i,j)\in X}h(g_A(\{i\}),g_B(\{j\}))
$$
is a finite observable. To prove that the family of all such $v_{.,.}$ forms a cocone under
the diagram $D_{A,B}$, let $(f_A,f_B)\colon (g_A,g_B)\to (g'_A,g'_B)$ be an arrow in $\elems{A}\times\elems{B}$.
We need to prove, for all $X\in\Dom(g_A)*\Dom(g_B)$,
$$v_{g'_A,g'_B}((f_A*f_B)(X))=v_{g_A,g_B}(X).$$
By (\ref{eq:coproduct2}) and the fact that $v_{g'_A,g'_B}$ is a 
morphism in $\EA$, the left-hand side expands to
\begin{align}
\notag
v_{g'_A,g'_B}((f_A*f_B)(X))=v_{g'_A,g'_B}(\sum_{(i,j)\in X}f_A(\{i\})\times f_B(\{j\}))=\\
\label{eq:contd}
\sum_{(i,j)\in X}v_{g'_A,g'_B}(f_A(\{i\})\times f_B(\{j\})).
\end{align}
For every $(i,j)\in X$,
\begin{align*}
v_{g'_A,g'_B}(f_A(\{i\})\times f_B(\{j\}))=
\sum_{(k,l)\in f_A(\{i\})\times f_B(\{j\})} h(g'_A(\{k\}),g'_B(\{l\}))=\\
h\bigl(\sum_{k\in f_A(\{i\})}g'_A(\{k\}),\sum_{l\in f_B(\{j\})}g'_B(\{l\})\bigr)=
h(g_A(\{i\}),g_B(\{j\})).
\end{align*}
Continuing the computation (\ref{eq:contd}),
\begin{align*}
\sum_{(i,j)\in X}v_{g'_A,g'_B}(f_A(\{i\})\times f_B(\{j\}))=
\sum_{(i,j)\in X}h(g_A(\{i\}),g_B(\{j\}))=v_{g_A,g_B}(X)
\end{align*}
In this way, every $C$-valued bimorphism gives us a cocone under $D_{A,B}$ with apex $C$.

Let $(v_{.,.})$ be a cocone under $D_{A,B}$. For $(a,b)\in A\times B$, put $h(a,b)=v_{\obs{a},\obs{b}}(\{(1,1)\})$.
We claim that $h$ is a bimorphism.

There is a unique morphism of Boolean algebras 
$u\colon 2^{[2]}\to 2^{[1]}$ with $u(\{1\})=\{1\}$ that makes both diagrams
\begin{equation}
\label{diag:u}
\xymatrix{
2^{[1]}
	\ar[r]^{\obs{\emptyset}}
&
A
\\
2^{[2]}
	\ar[u]^u
	\ar[ur]_{\obs{1}}
}
\qquad
\xymatrix{
2^{[1]}
	\ar[r]^{\obs{\emptyset}}
&
B
\\
2^{[2]}
	\ar[u]^u
	\ar[ur]_{\obs{1}}
}
\end{equation}
commute.
This implies the commutativity of the diagram
\begin{equation}
\label{diag:u2}
\xymatrix{
2^{[1]}*2^{[1]}
	\ar[r]^-{v_{\obs{\emptyset},\obs{\emptyset}}}
&
C
\\
2^{[2]}*2^{[2]}
	\ar[u]^{u*u}
	\ar[ur]_{v_{\obs{1},\obs{1}}}
}
\end{equation}
Note that $2^{[1]}*2^{[1]}\simeq 2^{[1]}$ and $2^{[1]}$ is initial in $\EA$, 
hence $v_{\obs{\emptyset},\obs{\emptyset}}=\obs{\emptyset}$. Using (\ref{diag:u2}) we may now
compute
$$
h(1,1)=v_{\obs{1},\obs{1}}(\{(1,1)\})=
v_{\obs{\emptyset},\obs{\emptyset}}((u*u)(\{(1,1)\}))=v_{\obs{\emptyset},\obs{\emptyset}}(\{(1,1)\})=1
$$
Let $a_1,a_2\in A$, $b\in B$.
Let $f_1,f_2,f_{1,2}$ be exactly as in the proof of Theorem \ref{thm:density}, diagram
(\ref{diag:summation}).
If we pair the observables in (\ref{diag:summation}) with the observable $\obs{b}\colon 2^{[2]}\to B$,
we obtain a commutative diagram in $\elems{A}\times\elems{B}$ that gives rise to the
following
part of the cocone $v_{.,.}$:
\begin{equation}
\label{diag:summation2}
\xymatrix@C=4pc{
~
&
2^{[2]\times[2]}
	\ar[ld]_{f_1*\id}
	\ar[rd]^{v_{\obs{a_1},\obs{b}}}
\\
2^{[3]\times[2]}
	\ar[rr]^{v_{\obs{a_1,a_2},\obs{b}}}
&
~
&
A
\\
&
2^{[2]\times[2]}
	\ar[lu]^{f_2*\id}
	\ar[ru]_{v_{\obs{a_2},\obs{b}}}
\\
~
&
2^{[2]\times[2]}
	\ar@/^1.2pc/[luu]^{f_{1,2}*\id}
	\ar@/_1.5pc/[ruu]_{v_{\obs{a_1+a_2},\obs{b}}}
}
\end{equation}

Note that 
\begin{align*}
(f_{1,2}*\id)\bigl(\{(1,1)\}\bigr)=\{(1,1),(2,1)\}=\{(1,1)\}+\{(2,1)\}=\\(f_1*\id)\bigl(\{(1,1)\}\bigr)+(f_2*\id)\bigl(\{(1,1)\}\bigr)
\end{align*}
and we can compute
\begin{align*}
h(a_1+a_2,b)=&v_{\obs{a_1+a_2},\obs{b}}\bigl(\{(1,1)\}\bigr)=\bigl(v_{\obs{a_1,a_2},\obs{b}}\circ (f_{1,2}*\id)\bigr)\bigl(\{(1,1)\}\bigr)=\\
&v_{\obs{a_1,a_2},\obs{b}}\bigl((f_1*\id)\bigl(\{(1,1)\}\bigr)+(f_2*\id)\bigl(\{(1,1)\}\bigr)\bigr)=\\
&v_{\obs{a_1,a_2},\obs{b}}\bigl((f_1*\id)\bigl(\{(1,1)\}\bigr)\bigr)+v_{\obs{a_1,a_2},\obs{b}}\bigl((f_2*\id)\bigl(\{(1,1)\}\bigr)\bigr)=\\
&v_{\obs{a_1},\obs{b}}\bigl(\{(1,1)\}\bigr)+v_{\obs{a_2},\obs{b}}\bigl(\{(1,1)\}\bigr)=h(a_1,b)+h(a_2,b).
\end{align*}

The proof of the right additivity of $h$ is analogous.

We should now check that this one-to-one correspondence between cocones under $D_{A,B}$ and
the objects of $\beta_{A,B}$ is functorial and the functors are mutually inverse. This
part of the proof is very straightforward and is thus omitted.
\end{proof}
\begin{corollary}
\label{coro:tensoriscolimit}
For every pair $A,B$ of effect algebras,
$$
A\otimes B=\varinjlim D_{A,B}
$$
\end{corollary}

\begin{theorem}
\label{thm:tensoriskan}
The tensor product of effect algebras is a functor $\EA\times\EA\to\EA$ that arises
as a left Kan extension of the functor $E\circ*\colon \FinBool\times\FinBool\to\EA$ along the
inclusion $E\times E\colon \FinBool\times\FinBool\to\EA\times\EA$.
\begin{equation}
\label{eq:kan}
\xymatrix@C4em@R4em{
\EA\times \EA
	\ar[rrd]^{\otimes}
&
~
\\
\FinBool\times\FinBool
	\ar[u]^{E\times E}
	\ar[r]_-{*}
&
\FinBool
	\ar[r]_{E}
	\ultwocell<\omit>{<-2>\eta}
&
\EA
}
\end{equation}
\end{theorem}
\begin{proof}
By (a dual version of) \cite[Theorem X.3.1]{mac1998categories}, we can express
the value of this Kan extension at $(A,B)$
as a colimit of a functor
$$
(\FinBool\times\FinBool\downarrow(A,B))\xrightarrow{Q}(\FinBool\times\FinBool)\xrightarrow{*}\EA\,,
$$
where $Q$ is the projection functor.
The comma category $\FinBool\times\FinBool\downarrow(A,B)$ is isomorphic to
$\elems{A}\times\elems{B}$ and the functor $*\circ Q$ is just the $D_{A,B}$ functor.
The rest follows by Corollary \ref{coro:tensoriscolimit}.
\end{proof}
\begin{corollary}
For any two finite Boolean algebras $A,B$, 
$$A*B\simeq A\otimes B.$$
\end{corollary}
\begin{proof}
By \cite[Remark 6.1.2]{riehl2016category}, the unit $\eta$ of the Kan extension (\ref{eq:kan}) is an isomorphism, because the functor $E\times E$ is full and faithful.
\end{proof}

It was proved by Day in \cite{day1970onclosed} that for every 
monoidal category $(\C,\square,I)$, the monoidal
structure can be extended to the category 
$[\C^{op},\Set]$ of presheaves on $\C$ by the rule
$$
X\Day Y=\int^{(c_1,c_2)}\C^{op}(c_1\square c_2,c)\times X(c_1)\times Y(c_2).
$$
\begin{theorem}
For every pair $A,B$ of effect algebras,
$$
A\otimes B\simeq L(R(A)\Day R(B))
$$
\end{theorem}
\begin{proof}
In our case,
$$
R(A)\Day R(B)=\int^{(c_1,c_2)}
	\FinBool(\_,c_1*c_2)\times\EA(c_1,A)\times\EA(c_2,B).
$$
Since $L$ is a left adjoint, $L$ preserves coends, so we may write
$$
L(R(A)\Day R(B))\simeq\int^{(c_1,c_2)}
	L(\FinBool(\_,c_1*c_2)\times\EA(c_1,A)\times\EA(c_2,B)),
$$
In detail, the category of elements
\begin{equation}
el(c_1,c_2)=\int \FinBool(\_,c_1*c_2)\times\EA(c_1,A)\times\EA(c_2,B)
\end{equation}
is a category with objects $(t,g_A,g_B)$, where $t\colon c\to c_1*c_2$, $g_A\colon c_1\to A$
and $g_B\colon c_2\to B$ and morphism being given by precomposition in the first variable:
$h\colon (t,g_A,g_B)\to (s,g_A,g_B)$ is simply the fact that $s\circ h=t$.
It is obvious that the morphisms in $el(c_1,c_2)$ preserve the pair $(g_A,g_B)$, so
$el(c_1,c_2)$ consists of pairwise isomorphic disjoint parts indexed by the elements of
the set $\EA(c_1,A)\times\EA(c_2,B)$. Note that each of the disjoint parts is isomorphic
to $\elems{c_1*c_2}$. In other words, 
$$
el(c_1,c_2)\simeq (\EA(c_1,A)\times\EA(c_2,A))\cdot \int R(c_1*c_2),
$$

where the dot denotes the copower.
To compute the value of $L$ means to take a colimit of a functor 
$D\colon el(c_1,c_2)\to\EA$ that maps every triple $(t,g_A,g_B)$ to the domain of
$t$.
Since $D$ behaves exactly like $D_{c_1*c_2}$ on every copy of $\int R(c_1*c_2)$
and since (by Theorem~\ref{thm:density}) $\varinjlim D_{c_1*c_2}\simeq c_1*c_2$,
we see that
$$
\varinjlim D\simeq (\EA(c_1,A)\times\EA(c_2,B))\cdot(c_1*c_2)
$$
and thus
\begin{align*}
L(R(A)\Day R(B))\simeq \int^{(c_1,c_2)}
	(\EA(c_1,A)\times\EA(c_2,B))\cdot(c_1*c_2)\\
	\simeq \int^{(c_1,c_2)}(\EA\times\EA)((c_1,c_2),(A,B))\cdot(c_1*c_2).
\end{align*}
By \cite[Theorem X.4.1]{mac1998categories}, this means that 
the functor $(A,B)\mapsto L(R(A)\otimes R(B))$ is
a left Kan extension of the functor $E\circ *\colon \FinBool\times\FinBool\to\EA$ 
along the inclusion $E\times E\colon \FinBool\times\FinBool\to\EA\times\EA$.
The rest follows by Theorem \ref{thm:tensoriskan}.
\end{proof}


\end{document}